\documentclass[11pt, a4paper, twoside]{article}
\pdfoutput=1
\usepackage{lspaper}
\usepackage{mfabacus}
\usepackage{tikzsymbols}
\usepackage{booktabs}
\usepackage{multirow, multicol}
\usepackage{placeins}
\usepackage{amssymb,amsmath}
\usepackage{pdflscape,afterpage}
\usetikzlibrary{matrix,arrows,decorations.pathmorphing,decorations.pathreplacing}
\usetikzlibrary{positioning,intersections,shapes.geometric,calc}
\usetikzlibrary{decorations.markings}
\RequirePackage{pxfonts}
\usepackage[toc,page]{appendix}
\usepackage{scalefnt,caption}
\usepackage[all, knot]{xy}

\renewcommand\spe[1]{\rspe{#1}}
\renewcommand\D[1]{\rD{#1}}
\newcommand{\vn}{\varnothing}

\begin{document}

\title{Schurian-finiteness of blocks of type $A$ Hecke algebras II}
\author{\begin{tabular}{cc}Sin\'ead Lyle&Liron Speyer\\[3pt]
{\normalsize University of East Anglia,}&{\normalsize Okinawa Institute of Science and Technology,}\\
{\normalsize Norwich Research Park, Norwich NR4 7TJ, UK}&{\normalsize Tancha, Okinawa, Japan 904-0495}\\[3pt]
{\normalsize\texttt{\normalsize s.lyle@uea.ac.uk}}&{\normalsize\texttt{\normalsize liron.speyer@oist.jp}}
\end{tabular}}

\renewcommand\auth{Sin\'ead Lyle \& Liron Speyer}

\runninghead{Schurian-finiteness of blocks of type $A$ Hecke algebras II}


\toptitle

\begin{abstract}
For any algebra $A$ over an algebraically closed field $\bbf$, we say that an $A$-module $M$ is Schurian if $\End_A(M) \cong \bbf$.
We say that $A$ is Schurian-finite if there are only finitely many isomorphism classes of Schurian $A$-modules, and Schurian-infinite otherwise.
In this paper, we build on the work of Ariki and the second author to show that all blocks of type $A$ Hecke algebras of weight at least $2$ in quantum characteristic $e \geq 3$ are Schurian-infinite.
This proves that if $e \geq 3$ then blocks of type $A$ Hecke algebras are Schurian-finite if and only if they are representation-finite.
\end{abstract}

\section{Introduction}

We let $\bbf$ denote an algebraically closed field of characteristic $p\geq 0$ throughout.
For a finite-dimensional algebra $A$ over $\bbf$, a left $A$-module $M$ is called Schurian if $\End_A(M) = \bbf$.
An algebra $A$ is called Schurian-finite if the number of isomorphism classes of Schurian $A$-modules is finite.
Recent results of Demonet, Iyama, and Jasso \cite[Theorem~4.2]{DIJ} tell us that Schurian-finiteness coincides with $\tau$-tilting finiteness, opening up new avenues for proving that an algebra is Schurian-finite or Schurian-infinite.
Typically, studies of $\tau$-tilting finiteness focus on bound quiver algebras and so do not study algebras for which a bound quiver presentation is difficult, as is the case for most `natural' algebras whose representation theory is studied.

In the prequel to this paper~\cite{as22}, Ariki and the second author initiated the study of Schurian-finiteness for Hecke algebras of symmetric groups.
Since blocks of weight $0$ and $1$ are known to have finite representation type, and therefore are Schurian-finite, blocks of weight at least $2$ were studied.
In quantum characteristic $e \geq 3$ it was proved that all weight $2$ and $3$ blocks are Schurian-infinite in any characteristic, as are many families of blocks of weight at least $4$, including all principal blocks.
It was conjectured in \cite[Conjecture~6.9]{as22} that \emph{all} blocks of weight at least $2$ are Schurian-infinite when $e \geq 3$.
In this paper, we build on the methods developed there to prove this conjecture; in other words we prove that all blocks of weight at least $4$ are Schurian-infinite in any characteristic when $e \geq 3$.
Hence we prove the following result.

\begin{thmx} \label{thm:all}
Every block of weight at least $2$ of type $A$ Hecke algebras in quantum characteristic at least $3$ is Schurian-infinite.
\end{thmx} 

Combined with the results of \cite{as22} which prove the result for weights $w=2$ and $w=3$, Theorem A follows immediately from the four  \cref{thm:mostblocks,thm:mostblocks32,thm:w4main,thm:wlargemain}.
\cref{thm:mostblocks} proves all cases of weight $w \geq 4$ except the exceptional case $e=3$ and $p=2$.
The same proof actually proves many cases when $e=3$, $p=2$, given in \cref{thm:mostblocks32}.
\cref{thm:w4main} deals with the case that $e=3$, $p=2$ and $w=4$.
Finally, \cref{thm:wlargemain} shows that the theorem holds when $e=3$, $p=2$ and $w \geq 5$.

\cref{thm:all} is the best that we could have hoped for when $e\neq 2$, since the blocks of type $A$ Hecke algebras have finite representation type in quantum characteristic $e=0$, or if they are of weight $w=0$ or $w=1$, by \cite{enreptype}.
Hence we have the following result.

\begin{thmx}\label{cor:repfin}
In quantum characteristic $e\neq 2$, blocks of type $A$ Hecke algebras are Schurian-finite if and only if they have finite representation type.
\end{thmx}

The paper is organised as follows.
In \cref{sec:background} we will give a brief background of some of our notational conventions and some useful abacus combinatorics, such as Scopes equivalences and their graded analogue proved in the prequel.
We also recall row- and column-removal for graded decomposition matrices.
However, a more detailed background on the representation theory of Hecke algebras, their graded decomposition numbers, and the related combinatorics, may be found in \cite{as22}.
Most importantly, we recall \cite[Proposition~2.18]{as22}, which allows us to deduce Schurian-infiniteness of a block in characteristic $p$ from the existence of a certain submatrix of the graded decomposition matrix in characteristics $0$ and $p$.

\cref{sec:main} contains the main results of the paper.
We begin with the case where $e\neq 3$ or $p\neq 2$.
Here we are able to quickly prove our main result, essentially by applying row-removal to reduce the problem to weight $2$, where we may easily find the necessary submatrices of the graded decomposition matrix using the results of \cite{as22}.
This trick doesn't work if $e=3$ and $p=2$, since the reduction by row-removal may land in the weight $2$ Rouquier block, for which no such submatrix exists.
Once we rule out the cases where this doesn't happen, so that the same argument may go through unchanged, we prove that weight $4$ Rouquier blocks, as well as three blocks that are in `nearby' Scopes classes, are Schurian-infinite, by directly computing applicable submatrices using the Jantzen sum formula and a known description of decomposition numbers in Rouquier blocks.
Finally, we use these results to prove that all weight $4$ blocks are Schurian-infinite, and rely on these same submatrices to prove our main result for weight at least $5$, in each case by using row-removal to reduce down to some weight $2$, $3$, or $4$ case.

\begin{ack}
The first author was supported by an LMS Emmy Noether Fellowship grant which enabled her to visit OIST where most of this research took place.
We thank the LMS for their support and OIST for their hospitality.
The second author is partially supported by JSPS Kakenhi grant number 20K22316.
We thank Susumu Ariki for helpful comments on an early draft, and Matt Fayers for his GAP code and for making his LaTeX style file for abacus displays publicly available.
\end{ack}

\section{Background}\label{sec:background} 

For brevity, we refer to \cite[Section~2]{as22} for the necessary standard background on partitions, Hecke algebras, and their graded decomposition numbers.
We will be interested in Hecke algebras of \emph{quantum characteristic} $e\geq 3$.
However, since we will deal with blocks of arbitrarily large weight, we must employ some new conventions for partitions and their abacus displays.

Suppose $\la = (\la_1, \la_2, \dots, \la_k)$ is a partition, with $\la_k >0$, and let $r \geq k$.
Then for $1\leq i \leq r$, we define the beta numbers $\beta_i = \la_i - i + r$.
We may then draw the $e$-runner abacus display with $r$ beads by marking positions on $e$ runners, starting with $0$ in the top left position and adding a bead at position $\beta_i$ for all $1 \leq i \leq r$.

An $e$-core $\rho$ is a partition whose abacus display has all beads pushed as far up their runners as possible, as in the example below.
The abacus display for a weight $w$ partition $\la$ can be obtained from that of its core $\rho$ by sliding beads down their runners, for a total of $w$ bead slides.

It is well-known that blocks of the Hecke algebra $\hhh$ are indexed by a core $\rho$ and a weight $w\geq 0$, and we will denote such a block by $B(\rho,w)$.

If $B$ is a block with core $\rho$, we take an abacus display for $\rho$ and define the integers $p_{e-1} > p_{e-2} > \dots > p_1 > p_0$ so that each is the position of the lowest bead on one of the runners.
We will use an $e$-quotient notation for our partitions, adapted so that it follows the ordering on $p_i$ as above.
Ordering the runners starting with the runner containing the $p_{e-1}$ position, and then the runner containing the $p_{e-2}$ position, and so on until the $p_0$ runner, we may read a partition from each runner, considered as a $1$-runner abacus display.
We will denote the empty partition by $\varnothing$, and use the shorthand notation $\varnothing^k$ to denote a string of $k$ components, each equal to the empty partition, in the $e$-quotient notation for a partition $\la$.

\begin{eg}
Let $e=4$, $\rho = (5,2^2)$, $\la = (9^2,3^2,1)$, and take $r=7$.
Then $\rho$ has beta numbers $(11,7,6,3,2,1,0)$, while $\la$ has beta numbers $(15,14,7,6,3,1,0)$, which are displayed on an abacus as below.
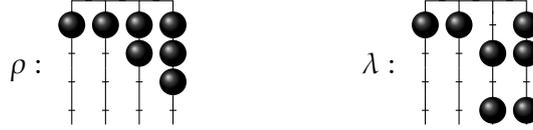
\[
\rho: \; \abacus(lmmr,bbbb,nnbb,nnnb,nnnn)
\qquad\qquad\qquad
\la: \; \abacus(lmmr,bbnb,nnbb,nnnn,nnbb)
\]
The partition $\la$ has weight $4$.
In the abacus display for its core, $\rho$, we can read off the integers $p_0 = 0$, $p_1 = 1$, $p_2 = 6$, and $p_3 = 11$.
In $e$-quotient notation, we write $\la = ((1),(2,1), \varnothing^2)$.
\end{eg}

While the $e$-quotient is usually only well-defined up to some cyclic permutation, our variant that orders the runners in terms of the integers $p_i$ is unique.
In other words, the $e$-quotient notation for a partition $\la$ does not change if we choose a different $r$, though the abacus display itself does.

Many of our results use partial decomposition matrices.
We denote by $d^{e,p}_{\la\mu}(v)$ the graded multiplicity of the simple module $\D\mu$ in the Specht module $\spe\la$, where the Hecke algebra is of quantum characteristic $e$ and the underlying field has characteristic $p \geq 0$.
We will capitalise on the following row- and column-removal theorems frequently in order to reduce the computation of graded decomposition numbers to those in weights 3 and 4.

\begin{thmc}{cmt02}{Theorem 1}\label{thm:rowremFock}
Let $\la = (\la_1, \la_2, \dots,\la_r)$ and $\mu = (\mu_1, \mu_2, \dots, \mu_s)$.
\begin{enumerate}[label=(\roman*)]
\item
If $\la_1 = \mu_1$, let $\bar\la = (\la_2, \dots, \la_r )$ and $\bar\mu =  (\mu_2, \dots, \la_s)$.
Then $d_{\la\mu}^{e,0}(v) = d_{\bar\la \bar\mu}^{e,0}(v)$.

\item
If $r=s$ and $\la_r, \mu_r >0$ (i.e.~$\la_1' = \mu_1'$), let $\bar\la = (\la_1-1, \la_2-1, \dots, \la_r-1)$ and $\bar\mu = (\mu_1-1, \mu_2-1, \dots, \mu_r-1)$.
Then $d_{\la\mu}^{e,0}(v) = d_{\bar\la \bar\mu}^{e,0}(v)$.
\end{enumerate}
\end{thmc}

\begin{thmc}{Donkin}{4.2(9) and 4.2(15)}\label{thm:rowremDecomp}
If $\la$, $\mu$, $\bar\la$, and $\bar\mu$, are as in either case above, then $d_{\la\mu}^{e,p}(1) = d_{\bar\la \bar\mu}^{e,p}(1)$.
\end{thmc}

We will make heavy use of the following graded Scopes equivalences, which generalise the equivalences introduced by Scopes~\cite{scopes} for symmetric groups and by Jost~\cite{jost} for their Hecke algebras.

When applying the Scopes equivalence, we number the runners of a chosen abacus display $0, 1, \dots, e-1$, so that position $i$ is on runner $i$ for each $i=0,1,\dots,e-1$.

Suppose that $B$ is a block of $\hhh$ of weight $w$, and core $\rho$, and that for some $i$ an abacus display for $\rho$ (or equivalently, for any partition in $B$) has $k$ more beads on runner $i$ than on runner $i-1$, for some $k\geq w$.
Let $A$ be the block of $\hhh[n-k]$ of weight $w$ and core $\Phi(\rho)$, whose abacus display is obtained from that of $\rho$ by swapping runners $i$ and $i-1$.
We may define this map $\Phi$ in the same way for any partition $\la \in B$.
That is, we swap runners $i$ and $i-1$ of the corresponding abacus display for $\la$, yielding a partition $\Phi(\la) \in A$.

If $i=0$, we actually want to swap runners $0$ and $e-1$, and in doing so we need $k+1$ more beads on runner $0$ than runner $e-1$.
We will favour changing $r$ to avoid the need for this exceptional treatment.

Scopes showed that this $\Phi$ is a bijection from $B$ to $A$, that maps $e$-regular partitions to $e$-regular partitions.
We will use the following graded lift of Scopes's and Jost's result.

\begin{propc}{as22}{Proposition~3.2}\label{prop:scopes}
If $B$ and $A$ are blocks of $\hhh$ and $\hhh[n-k]$ as above, with $k\geq w$, then they are graded Morita equivalent, and the equivalence is realised via $\D\la \leftrightarrow \D{\Phi(\la)}$.
In particular, if $\la$ and $\mu$ are partitions of $n$, with $\mu$ $e$-regular, then for any $p\geq0$, $d_{\la\mu}^{e,p}(v) = d_{\Phi(\la)\Phi(\mu)}^{e,p}(v)$.
\end{propc}

The following proposition is our main tool for proving the results of this paper.
It allows us to deduce that a block of $\hhh$ is Schurian-infinite from a submatrix of its graded decomposition matrix.

\begin{propc}{as22}{Proposition~2.18}\label{prop:matrixtrick}
Suppose that $e\geq3$ and $\bbf$ has characteristic $p\geq 0$.
If a submatrix of the graded decomposition matrix in characteristic $0$ is one of the following matrices, and $d_{\la\mu}^{e,p}(1) = d_{\la\mu}^{e,0}(1)\in \{0,1\}$ holds, for all $e$-regular partitions $\la, \mu$ that label rows of the submatrix, then the block in which those partitions belong is Schurian-infinite.

\begin{multicols}{3}
\begin{equation*}\label{targetmatrix}\tag{\(\dag\)}
\begin{pmatrix}
1\\
v & 1\\
0 & v & 1\\
v & v^2 & v & 1
\end{pmatrix}
\end{equation*}
\begin{equation*}\label{targetmatrixalt}\tag{\(\ddag\)}
\begin{pmatrix}
1\\
v & 1\\
v & 0 & 1\\
v^2 & v & v & 1
\end{pmatrix}
\end{equation*}
\begin{equation*}\label{targetmatrixstar}\tag{\(\spadesuit\)}
\begin{pmatrix}
1\\
0 & 1\\
v & v & 1\\
0 & v^2 & v & 1\\
v^2 & 0 & v & 0 &1
\end{pmatrix}
\end{equation*}
\end{multicols}
\end{propc}

\section{Main Results}\label{sec:main}

It was shown in \cite{as22} that weight $2$ blocks are Schurian-infinite, by breaking into five cases depending on the the comparative values of $p_{e-1}$, $p_{e-2}$, and $p_{e-3}$ for the block's core $\rho$.
When $p_{e-1} - p_{e-2} >e$, handled in \cite[Sections~4.3 and 4.5]{as22}, the methods in characteristic $2$ did not directly appeal to \cref{prop:matrixtrick}, as no such submatrices are available in such cases.
Instead, results of Fayers about extensions between simples \cite{faywt2} were used directly, and when $p_{e-2} - p_{e-3} > e$ and $e=3$, separate ad hoc methods were required.

We will initially avoid the case $e=3$, $p=2$, and thus have a bit more flexibility to work with.
The following result essentially deduces the same result as \cite[Sections~4.3 and 4.5]{as22} -- that $B(\rho,2)$ is Schurian-infinite when $p_{e-1} - p_{e-2} >e$ -- using \cref{prop:matrixtrick} directly, under the additional assumption that we're not in this difficult $e=3$, $p=2$ case.
This will be useful to quickly prove that all blocks of weight at least $4$ are Schurian-infinite, in conjunction with those submatrices already determined in \cite[Section~4]{as22}.

\begin{lem}\label{lem:wt2rouquish}
Suppose $e\geq 4$ and $\rho$ is a core satisfying $p_{e-1} - p_{e-2} >e$.
Define four partitions as follows.
\[
\la^{(1)} = ((1),(1),\varnothing^{e-2}), \qquad \la^{(2)} = ((1),\varnothing,(1),\varnothing^{e-3}).
\]
\begin{enumerate}[label=(\roman*)]
\item
If $p_{e-2} - p_{e-3} <e$, then define 
\[
\la^{(3)} = (\varnothing,(2),\varnothing^{e-2}), \qquad \la^{(4)} = (\varnothing,\varnothing,(2),\varnothing^{e-3}).
\]

\item
If $p_{e-2} - p_{e-3} >e$, then define 
\[
\la^{(3)} = (\varnothing,(1^2),\varnothing^{e-2}), \qquad \la^{(4)} = (\varnothing,(1),(1),\varnothing^{e-3}).
\]
\end{enumerate}
Then the four partitions $\la^{(1)}$, $\la^{(2)}$, $\la^{(3)}$, and $\la^{(4)}$ give (\ref{targetmatrixalt}) as a submatrix of the graded decomposition matrix in any characteristic.
It follows from \cref{prop:matrixtrick} that $B(\rho,2)$ is Schurian-infinite.
\end{lem}

\begin{proof}
We essentially argue as in the proofs in \cite[Section~4]{as22}.
We may first apply \cite[Theorem 3.2]{jm02} $e-4$ times to determine $d_{\la\mu}^{e,0}$ by examining the leading four runners alone.
One may check that up to Scopes equivalence, the remaining partitions live in one of five blocks, with cores $(3)$, $(4,1^2)$, and $(4,1^3)$ satisfying $p_{e-2} - p_{e-3} <e$, and cores $(5,2^2)$ and $(6,3^2,1^3)$ satisfying $p_{e-2} - p_{e-3} >e$.
Since we are in a block of weight 2, the decomposition matrix in characteristic $p$ for $p \geq 3$ is identical to the decomposition matrix in characteristic $0$.
Hence in each case, we may easily verify, for example by the LLT algorithm \cite{LLT}, that the corresponding submatrices are each (\ref{targetmatrixalt}) if $p\neq 2$.
By applying \cite[Corollary 2.4]{faywt2}, we may check that these submatrices are in fact identical when $p=2$.
\end{proof}

\subsection{When $e\neq 3$ or $p \neq 2$}

\begin{thm}\label{thm:mostblocks}
Let $e\geq 3$, and $w\geq 4$, and let $\rho$ be an $e$-core.
Unless $e=3$ and $p = 2$, the weight $w$ block $B(\rho,w)$ of $\hhh$ is Schurian-infinite.
\end{thm}

\begin{proof}
Our method of proof will be to reduce to weight 2 and then apply the results of \cite[Section~4]{as22}.
In order to do this, we will slide the lowest bead on the abacus display of $\rho$ down $w-2$ spaces on each of four partitions (five in case (v)), and then use the row-removal results of \cref{thm:rowremFock,thm:rowremDecomp} to remove this bead (i.e.~the first row of each of the four partitions).
To then apply results of \cite{as22}, we must know the position of the bead above the lowest bead on the abacus display for $\rho$, in order to know the new ordering on runners.
If $p_{e-1}-p_{e-2} > e$, then the ordering on runners is unchanged -- this is reflected in our first seven cases below, five of which are essentially identical to the five subsections of \cite[Section~4]{as22} once we've removed the first row, while the other two are handled by \cref{lem:wt2rouquish}.

Thus we will list the four partitions required in each case below.
The proof then follows by applying the known decomposition numbers in weight 2, where our chosen partitions were shown in \cite{as22} and \cref{lem:wt2rouquish} to satisfy the conditions of \cref{prop:matrixtrick}.

We summarise our results in Table \ref{tab1}, where the following shorthand is used: 
\begin{align*}
&\alpha \text{ denotes } p_i -p_j <e, && \beta \text{ denotes } e<p_i -p_j <2e, \\
 &\gamma \text{ denotes } 2e < p_i-p_j,&& \delta \text{ denotes } e< p_i -p_j.\qedhere
\end{align*}

\afterpage{%
    \clearpage
    \thispagestyle{empty}
\begin{landscape}
\begin{center}
\renewcommand{\arraystretch}{1.2}
\[\begin{array}{|m{10pt}m{10pt}m{10pt}|m{15pt}m{15pt}|m{10pt}m{10pt}m{10pt}|m{15pt}m{15pt}|ll|ll|l} \hline
\multicolumn{3}{|c|}{p_{e-1}-p_{e-2}} & \multicolumn{2}{c|}{p_{e-2}-p_{e-3}} & 
\multicolumn{3}{c|}{p_{e-1}-p_{e-3}} & \multicolumn{2}{c|}{p_{e-1}-p_{e-4}} & \multicolumn{4}{c|}{} \\ 
$\alpha$ &$\beta$&$\gamma$ &\multicolumn{1}{c}{\alpha} &\multicolumn{1}{c|}{\delta} & $\alpha$ &  $\beta$& $\gamma$ & \multicolumn{1}{c}{\alpha} & \multicolumn{1}{c|}{\delta} & \multicolumn{4}{c|}{} \\ \hline 
&$\bullet$&&\multicolumn{1}{c}{\bullet}&&&$\bullet$&&&&((w-2,2), \vn^{e-1}) & ((w-2),(2),\vn^{e-2}) & \text{(\ref{targetmatrix})} &\text{\cite[\S~4.1]{as22}} \\
&&&&&&&&&&((w-2),\vn,(2),\vn^{e-3}) & ((w-2,1),\vn,(1),\vn^{e-3}) && \\ \hline 
&$\bullet$&&\multicolumn{1}{c}{\bullet} &&&&$\bullet$&&&((w-2,2), \vn^{e-1}) & ((w-2),(2),\vn^{e-2}) & \text{(\ref{targetmatrix})} &\text{\cite[\S~4.2]{as22}}\\
&&&&&&&&&&((w-2,1),(1),\vn^{e-2}) & ((w-2,1^2),\vn^{e-1}) &&\\ \hline 
&&$\bullet$&\multicolumn{1}{c}{\bullet}&&&&&& &((w-2,2), \vn^{e-1}) & ((w-2,1),(1),\vn^{e-2}) & \text{(\ref{targetmatrix})} &\text{\cite[\S~4.3]{as22}} \hspace{22pt} p\neq2\\
&&&&&&&&&&((w-2),(2),\vn^{e-2}) & ((w-2),\vn,(2),\vn^{e-3}) &  & \\ \hline 
&&$\bullet$&\multicolumn{1}{c}{\bullet}&&&&&& &((w-2,1),(1), \vn^{e-2}) & ((w-2,1),\vn,(1),\vn^{e-3}) &\text{(\ref{targetmatrixalt})} & \text{\cref{lem:wt2rouquish}} \hspace{28pt} p=2  \\
&&&&&&&&&&((w-2),(2),\vn^{e-2}) & ((w-2),\vn,(2),\vn^{e-3}) & &\hspace{82pt} e \geq 4 \\ \hline 
&$\bullet$ & & &$\bullet$ &&&&&& ((w-2,2),\vn^{e-1}) & ((w-2),(2),\varnothing^{e-2}) & \text{(\ref{targetmatrix})} &\text{\cite[\S~4.4]{as22}}\\
&&&&&&&&&&((w-2,1),(1),\vn^{e-2}) & ((w-2,1^2),\vn^{e-1}) && \\ \hline
&&$\bullet$ &&$\bullet$ &&&&&& ((w-2,2),\vn^{(e-1)}) & ((w-2,1^2),\vn^{e-1}) &  \text{(\ref{targetmatrixstar})} &\text{\cite[\S~4.5]{as22}} \hspace{22pt} p\neq 2\\
&&&&&&&&&&((w-2,1),(1),\vn^{e-2}) & ((w-2),(2),\vn^{e-2}) & &\\
&&&&&&&&&&((w-2),(1^2),\vn^{e-2}) &&&  \\ \hline 
&&$\bullet$ &&$\bullet$ &&&&&& ((w-2,1),(1),\vn^{(e-2)}) & ((w-2,1),\vn,(1),\vn^{e-3}) &  \text{(\ref{targetmatrixalt})} &\text{\cref{lem:wt2rouquish}} \hspace{27pt} p =2\\
&&&&&&&&&&((w-2),(1^2),\vn^{e-2}) & ((w-2),(1),(1),\vn^{e-3}) & &\hspace{81pt} e \geq 4 \\ \hline
&&&&&&&&\multicolumn{1}{c}{\bullet} && ((w-2),(2),\vn^{e-2}) & ((w-2),\vn,(2),\vn^{e-3}) & \text{(\ref{targetmatrix})} &\text{\cite[\S~4.1]{as22}} \\
&&&&&&&&&&((w-2),\vn^2,(2),\vn^{e-4}) & ((w-2),(1),\vn,(1),\vn^{e-4}) && \\ \hline
&&&&& $\bullet$&&&& \multicolumn{1}{c|}{\bullet} & ((w-2),(2),\vn^{e-2}) & ((w-2),\vn,(2),\vn^{e-3}) &  \text{(\ref{targetmatrix})} &\text{\cite[\S~4.1]{as22}}\\
&&&&&&&&&&((w-2,1),\vn^{e-1}) & ((w-2,1),(1),\vn^{e-2}) && \\ \hline 
$\bullet$ &&&\multicolumn{1}{c}{\bullet} &&& $\bullet$ &&&& ((w-2),(2),\vn^{e-2}) &((w-2,2),\vn^{e-1}) &  \text{(\ref{targetmatrix})} &\text{\cite[\S~4.1]{as22}}\\
&&&&&&&&&& ((w-2),\vn,(2),\vn^{e-3}) & ((w-2),(1),(1),\vn^{e-3}) && \\ \hline 
$\bullet$ &&& & $\bullet$ && && &&((w-2),(2),\vn^{e-2}) & ((w-2,2),\vn^{e-1}) &  \text{(\ref{targetmatrix})} &\text{\cite[\S~4.2 and \S~4.4]{as22}} \\ 
&&&&&&&&&&((w-2,1),(1),\vn^{e-2}) & ((w-2),(1^2),\vn^{e-3}) &&  \\ \hline 
\end{array}\]
\captionof{table}{A case by case analysis of \cref{thm:mostblocks}} \label{tab1}
\end{center}
\end{landscape}
    \clearpage
}
\end{proof}

\begin{eg}
Let $w=e=5$ and $p\geq 0$.
Take $\rho=(10,6,4,3,2^2,1^4)$ and consider the block $B(\rho,5)$.
The abacus display of $\rho$ is given by
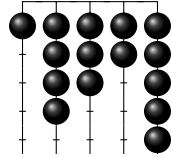
\[\abacus(lmmmr,bbbbb,nbbbb,nbbnb,nbnnb,nnnnb)\]
and we see that 
$p_0=0$, $p_1=8$, $p_2=12$, $p_3=16$ and $p_4=24$ so that
$e<p_{e-1} - p_{e-2} <2e$, $p_{e-2}-p_{e-3}<e$ and $p_{e-1}-p_{e-3}>2e$
and we are in the second case in Table \ref{tab1}.
We therefore take the partitions given by the following four abacus displays.
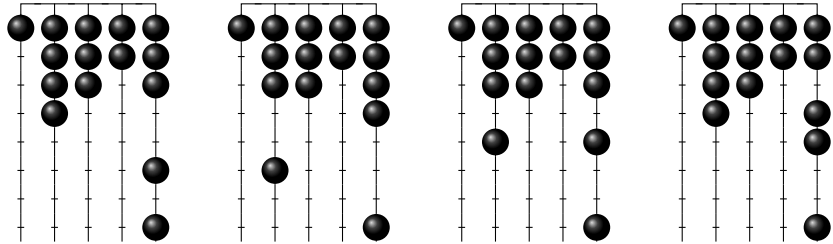
\[
\abacus(lmmmr,bbbbb,nbbbb,nbbnb,nbnnn,nnnnn,nnnnb,nnnnn,nnnnb) \qquad 
\abacus(lmmmr,bbbbb,nbbbb,nbbnb,nnnnb,nnnnn,nbnnn,nnnnn,nnnnb) \qquad 
\abacus(lmmmr,bbbbb,nbbbb,nbbnb,nnnnn,nbnnb,nnnnn,nnnnn,nnnnb) \qquad 
\abacus(lmmmr,bbbbb,nbbbb,nbbnn,nbnnb,nnnnb,nnnnn,nnnnn,nnnnb)
\]
To compute the partial decomposition matrix indexed by these partitions, we first apply row removal, as described in \cref{thm:rowremFock,thm:rowremDecomp}, to remove the first bead from each abacus configuration.
We are then in a block of weight 2 and we follow \cite[Section~4.1]{as22} to show that the partial decomposition matrix indexed by these partitions is (\ref{targetmatrix}) in characteristic $0$, and that the corresponding ungraded decomposition numbers are characteristic-free, so that applying \cref{prop:matrixtrick} yields that the block $B(\rho,5)$ is Schurian-infinite.
%
\end{eg}

\subsection{When $e=3$ and $p=2$}
We now consider the remaining cases.
For the remainder of this section, {\bf fix } $\boldsymbol{e=3}$ {\bf and } $\boldsymbol{p=2}$.

We will index weight $w$ blocks $B(\rho,w)$, with core $\rho$, by tuples $[1,s_1,s_2]$, where $\rho$ has an abacus display in which the $p_0$ is placed in the top left position, so that there is just one bead on the leftmost runner, while the middle and right-hand runners have $s_1\geq 1$ and $s_2 \geq 1$ beads on, respectively.
By applying the Scopes equivalences of \cref{prop:scopes}, we may further assume that $s_1 \leq w$ and $s_2 \leq s_1 + w - 1$.
In this way, the triples $[1,s_1,s_2]$ index Scopes classes of blocks of a fixed weight, and we thus identify the triples with the corresponding Scopes classes.
We will sometimes write $B([1,s_1,s_2],w)$ for the associated block, rather than $B(\rho,w)$.

Note that we may also make the following observation to simplify the task of proving that each $B(\rho,w)$ is Schurian-infinite.
If we replace the core $\rho$ with its conjugate, $\rho'$, then $B(\rho,w)$ is isomorphic to $B(\rho',w)$, with the isomorphism being induced by the $\#$-automorphism on $\hhh$.
This argument was used in \cite[Corollary~6.5]{as22}.
In such a situation, we will say that two Scopes classes are conjugate.

\begin{thm}\label{thm:mostblocks32}
Let $w\geq 4$, and let $\rho$ be an $3$-core.
Unless $\rho$ satisfies $p_{2} - p_{1} >2e$, the weight $w$ block $B(\rho,w)$ of $\hhh$ is Schurian-infinite.
\end{thm}

\begin{proof}
The proof is almost identical to that of \cref{thm:mostblocks}.
Cases (iii) and (v) there do not appear here, since we have assumed that $p_{2} - p_{1} < 2e$.
For all other cases, we may take exactly the same partitions as in that proof, and obtain the result in the same way.
\end{proof}

\begin{rem}
When the weight is $4$, the above theorem covers all Scopes classes except for the following nine:
\begin{alignat*}{3}
&[ 1, 1, 3 ], \qquad &&[ 1, 1, 4 ], \qquad &&[ 1, 2, 4 ]\\
&[ 1, 2, 5 ], &&[ 1, 3, 5 ], &&[ 1, 3, 6 ]\\
&[ 1, 4, 1 ], &&[ 1, 4, 6 ], &&[ 1, 4, 7 ].
\end{alignat*}
If we tried to argue by the same method, our row-removal to reduce to weight 2 would leave us computing in the weight 2 Scopes classes $[1,2,3]$ and $[1,1,2]$.
For the latter, no submatrix of the decomposition matrix will do what we need, but in \cite[Section~4.3]{as22} we argued directly by looking at extensions between simples.
One could instead argue that is it Schurian-infinite since the conjugate Scopes class $[1,2,2]$ is.
For the former -- the Rouquier block -- the class is self-conjugate, and the extensions do not suffice, so we had to find other means of directly proving that the block is Schurian-infinite.
\end{rem}

We next deal with some special cases when $w=4$.
Recall that the Rouquier block for weight $4$ is the block with Scopes class $[1,4,7]$.

\begin{prop}\label{prop:Rouqblocks}
Let $B(s,4)$ be the Rouquier block of weight $4$.
We define the following four partitions in $B(s,4)$ by their $3$-quotients.

\begin{align*}
\la^{(1)} &= ((1^{2}),(1^2),\varnothing), 
&\la^{(2)} &= ((1),(2,1),\varnothing),\\
\la^{(3)} &= ((1),(1^3),\varnothing), &
\la^{(4)} &= (\varnothing,(2,1^2),\varnothing).
\end{align*}
Then the partial decomposition matrix corresponding to these four partitions in both characteristic $0$ and characteristic $2$ is equal to (\ref{targetmatrixalt})
and hence $B(s,4)$ is Schurian-infinite.
\end{prop}

\begin{proof}

It may easily be checked, using either the LLT algorithm or the explicit formula for decomposition numbers for Rouquier blocks in characterstic 0 \cite[Corollary~10]{leclercmiyachi}, that the partial decomposition matrix is as described in characteristic $0$.
Recall \cite[Theorem~5.17]{bk09} that there exists a lower unitriangular matrix $A =(a_{\la\mu})$ - the graded adjustment matrix - with entries in $\mathbb{N}[v]$ and rows and columns indexed by the $e$-regular partitions such that
\begin{equation}
d^{e,p}_{\la\mu}(v) = d^{e,0}_{\la\mu}(v) + \sum_{\nu \triangleleft \mu} d^{e,0}_{\la\nu}(v)a_{\nu\mu}(v).\label{eqn:adj} 
\end{equation} 
Suppose $\nu,\mu \in B(\rho,4)$ have respective $3$-quotients $(\nu^{(0)},\nu^{(1)},\nu^{(2)})$ and $(\mu^{(0)},\mu^{(1)},\mu^{(2)})$.
Since $B(s,4)$ is a Rouquier block, by \cite[Proposition~4.4]{jlm} we have  $a_{\nu\mu}(v) = 0$ unless $|\nu^{(i)}|=|\mu^{(i)}|$ for $i=0,1,2$.
Putting these results together, we see that $d_{\la^{(i)}\la^{(j)}}^{3,2}(v) = d_{\la^{(i)}\la^{(j)}}^{3,0}(v)$ unless $j=2$, with
\[
d^{3,2}_{\la^{(i)} \la^{(2)}}(v) = d^{3,0}_{\la^{(i)}\la^{(2)}}(v) + d^{e,0}_{\la^{(i)}\la^{(3)}} (v) a_{\la^{(3)} \la^{(2)}}(v),
\]
so that the partial decomposition matrices agree if $a_{\la^{(3)}\la^{(2)}}=0$.
If $d^{3,2}_{\la^{(3)}\la^{(2)}}=0$ then certainly $a_{\la^{(3)} \la^{(2)}}=0$, so we find this decomposition number.
Note that there does not exist $\sigma \in B(s,4)$ with $\la^{(3)} \triangleleft \sigma \triangleleft \la^{(2)}$ so since the Jantzen coefficient $J_{\la^{(3)}\la^{(2)}}$ is equal to $0$ we also have $d^{3,2}_{\la^{(3)}\la^{(2)}}=0$.
(For more information on the Janzten sum formula, see \cite[Section~5.2]{Mathas}.)
\end{proof}

\begin{rem}
The above argument extends readily to proving that any Rouquier block for $e\geq 3$, $w\geq 4$ and $p \neq 3$ is Schurian-infinite (and there exist other partitions that will prove the case $p \neq 2$).
However, it is slightly cleaner, notationally, to handle this one case alone, and our other methods -- used to prove \cref{thm:mostblocks,thm:wlargemain} -- apply to a much broader collection of blocks.
\end{rem}

The following result holds for any parameters $e$ and $p$, although we only require it for our choice of $e=3$ and $p=2$.

\begin{lem} \label{lem:rest}
Let $\la,\mu \in B(\rho,w)$ be such that $\mu$ is $e$-regular and both $\la$ and $\mu$ have exactly $k$ removable $i$-nodes.
Let $\bar{\la}$ and $\bar{\mu}$ be the partitions obtained by removing $k$ nodes from $\la$ and $\mu$ respectively.
Suppose that $\bar{\la}$ and $\bar{\mu}$ both have exactly $k$ addable $i$-nodes.
Then $\bar{\mu}$ is $e$-regular and $d^{e,p}_{\la\mu}(1) \geq d^{e,p}_{\bar{\la}\bar{\mu}}(1)$.
\end{lem}

\begin{proof}
Apply the functor $i$-Res as described in \cite[Section~6.1]{Mathas} to the projective indecomposable module $P(\mu)$.
\end{proof}

\begin{prop}\label{prop:NearRouqblocks}
Suppose that $s \in \{[1,4,6], [1,3,6], [1,3,5]\}$.
We define the following four partitions in $B(s,4)$ by their $3$-quotients.
\begin{align*}
\la^{(1)} &= ((1^{2}),(1^2),\varnothing),
&\la^{(2)} &= ((1),(2,1),\varnothing),\\
\la^{(3)} &= ((1),(1^3),\varnothing),
&\la^{(4)} &= (\varnothing,(2,1^2),\varnothing).
\end{align*}
Then the partial decomposition matrix corresponding to these four partitions in both characteristic $0$ and characteristic $2$ is equal to (\ref{targetmatrixalt})
and hence $B(s,4)$ is Schurian-infinite.
\end{prop}

\begin{proof}
In fact, we prove the stronger claim that the proposition holds for all $s \in \{[1,4,7], [1,7,4],$ $[1,4,6], [4,1,7], [1,3,6], [1,6,3], [1,3,5]\}$.
We may easily verify using the LLT algorithm that the seven partial decomposition matrices are as stated in characteristic $0$, thus by Equation \ref{eqn:adj} we also have lower bounds on the entries in the matrices in characteristic $2$.
The case that $s=[1,4,7]$ is exactly \cref{prop:Rouqblocks}.
We first apply \cref{lem:rest} to obtain upper bounds on the entries for $[1,7,4]$, coming from those known decomposition numbers for the Rouquier block $[1,4,7]$, as above.
Since these agree with the lower bounds, the proposition holds for $[1,7,4]$; an identical argument shows it is true for $[4,1,7]$.
We apply Scopes equivalence to $[1,7,4]$ and $[4,1,7]$ respectively to show that it is true for $[7,1,4]=[1,4,6]$ and $[4,7,1]=[1,3,6]$.
We repeat the argument applying \cref{lem:rest} to $[1,3,6]$ to show the result holds for $[1,6,3]$; finally Scopes equivalence proves that it holds for $[6,1,3]=[1,3,5]$.
\end{proof}

We will need the following lemma in the proof of \cref{thm:wlargemain}

\begin{lem}\label{lem:wt3NearRouq}
Suppose $w=3$ and $s = [1,3,4]$ or $[1,2,3]$.
Define the following four partitions in $B(s,3)$ by their 3-quotients.
\begin{align*}
\la^{(1)} &= ((2),(1),\varnothing), 
&\la^{(2)} &= (\varnothing,(3),\varnothing),\\
\la^{(3)} &= ((1),(2),\varnothing), 
&\la^{(4)} &= ((1^2),(1),\varnothing).
\end{align*}
Then in both blocks, the submatrix corresponding to these four partitions is (\ref{targetmatrix}) in both characteristic $0$ and $2$.
\end{lem}

\begin{proof}
Since $w=3$, we can prove this result by first using to LLT algorithm to compute the entries of the submatrix in characteristic $0$ and then using the results of \cite{faytan06} to show that it does not change in characteristic $2$.
In both cases however, if we write out the four partitions and apply column removal, we see that our submatrix is equal to the submatrix corresponding to the (weight 2) partitions
\begin{align*}
\mu^{(1)} & = (7,1), 
&\mu^{(2)}& = (6,2),\\
\mu^{(3)} &= (4^2), 
&\mu^{(4)} &= (4,2^2),
\end{align*}
and this submatrix can be computed by hand.
\end{proof}

\begin{rem}
The submatrix used above was missed by the authors in \cite{as22}, and a more difficult proof was used.
For the Scopes class $[1,3,4]$, we deduced its Schurian-infiniteness from its transpose class $[1,2,4]$, for which Schurian-infiniteness was proved simultaneously with the Rouquier block $[1,3,6]$.
However, the above significantly shortens the proof that $[1,2,3]$ is a Schurian-infinite Scopes class.
\end{rem}

\begin{thm}\label{thm:w4main}
Suppose $e=3$, $p=2$, $w = 4$, and let $s$ be a Scopes class for $w$.
Then the block $B(s,w)$ of $\hhh$ is Schurian-infinite.
\end{thm}

\begin{proof}
All but nine Scopes classes are handled by \cref{thm:mostblocks32}, as discussed in the remark below it.
The Rouquier block $[1,4,7]$ is handled in \cref{prop:Rouqblocks}, while $[1,3,5]$, $[1,3,6]$, and $[1,4,6]$ are dealt with in \cref{prop:NearRouqblocks}.
The remaining Scopes classes $[1,1,3]$, $[1,1,4]$, $[1,2,4]$, $[1,2,5]$, and $[1,4,1]$ are conjugate to the classes $[1,3,3]$, $[1,4,4]$, $[1,3,4]$, $[1,4,5]$, and $[1,4,3]$, respectively, and are thus Morita equivalent to blocks we've already shown to be Schurian-infinite.
\end{proof}

\begin{thm}\label{thm:wlargemain}
Suppose $e=3$, $p=2$, $w \geq 5$, and let $s$ be a Scopes class for $w$.
Then the block $B(s,w)$ of $\hhh$ is Schurian-infinite.
\end{thm}

\begin{proof}
First, suppose the Scopes class is $[1,s_1,s_2]$, with $s_1 \leq s_2$.
Note that $[1,s_1,s_2]' = [1,s_2-s_1,s_2]$.
Since $B(s,w)$ and $B(s',w)$ are Morita equivalent, it is sufficient to consider the cases where $s_2-s_1 \leq s_1-1$ so we assume further that this identity holds.

If $s_2-s_1 \leq 1$, the result follows immediately by \cref{thm:mostblocks32}.

So suppose $s_2-s_1 = 2$, so that $s_1 \geq 3$.
Then we define partitions
\begin{align*}
\la^{(1)} &= ((w-3,2),(1),\varnothing), 
&\la^{(2)} &= ((w-3),(3),\varnothing),\\
\la^{(3)} &= ((w-3,1),(2),\varnothing), 
&\la^{(4)} &= ((w-3,1^2),(1),\varnothing).
\end{align*}
By row-removal \cref{thm:rowremFock,thm:rowremDecomp}, the relevant submatrix of the decomposition matrix matches that of the four partitions obtained by removing the first row from each, or in other words removing the lowest bead from the abacus display of each.
The remaining partitions are then in the block $B([1,s_1,s_1+1], 3)$, which is Scopes equivalent to $B([1,3,4], 3)$ for any $s_1 \geq 3$.
The remaining partitions are precisely those used in \cref{lem:wt3NearRouq}, and the result follows, by \cref{prop:matrixtrick}.

Next, suppose $s_2-s_1=3$, so that we also have $s_1 \geq 4$.
Then we define partitions
\begin{align*}
\la^{(1)} &= ((w-4,1^2),(1^2),\varnothing), 
&\la^{(2)} &= ((w-4,1),(2,1),\varnothing),\\
\la^{(3)} &= ((w-4,1),(1^3),\varnothing), 
&\la^{(4)} &= ((w-4),(2,1^2),\varnothing).
\end{align*}
Arguing as before, removing the first row yields partitions in the block $B([1,s_1,s_1+2], 4)$, which is Scopes equivalent to $B([1,4,6], 4)$ for any $s_1 \geq 4$.
The remaining partitions are precisely those used in \cref{prop:NearRouqblocks}, and the result follows.

Now we suppose that $s_2-s_1\geq 4$, so that we also have $s_1 \geq 5$.
Then, taking the exact same partitions as in the previous case, and performing row removal as before now yields partitions in the block $B([1,s_1,s_2-1], 4)$, which is Scopes equivalent to $B([1,4,7], 4)$.
Since the remaining partitions are still those used in \cref{prop:Rouqblocks}, the result follows once more.
This completes the proof for all Scopes classes $[1,s_1,s_2]$ for $s_2 \geq s_1$.

We now assume that $s_2<s_1$.
Note that $[1,s_1,s_2]' = [1,s_1,s_1-s_2]$, so it suffices to assume that $s_2 \geq s_1-s_2$.

If $s_1-s_2 \leq 2$, the result follows immediately by \cref{thm:mostblocks32}.

Next, suppose $s_1-s_2=3$, so that $s_2 \geq 3$.
Then we define partitions
\begin{align*}
\la^{(1)} &= ((w-3,2),(1),\varnothing), 
&\la^{(2)} &= ((w-3),(3),\varnothing),\\
\la^{(3)} &= ((w-3,1),(2),\varnothing), 
&\la^{(4)} &= ((w-3,1^2),(1),\varnothing).
\end{align*}
Arguing as before, removing the first row yields partitions in the block $B([1,s_1-1,s_2], 3)$, which is Scopes equivalent to $B([1,3,4], 3)$.
The remaining partitions are precisely those used in \cref{lem:wt3NearRouq}, and the result follows.

Finally, we suppose that $s_1-s_2\geq4$ so that $s_2 \geq 4$.
Then we define partitions
\begin{align*}
\la^{(1)} &= ((w-4,1^2),(1^2),\varnothing), 
&\la^{(2)} &= ((w-4,1),(2,1),\varnothing),\\
\la^{(3)} &= ((w-4,1),(1^3),\varnothing), 
&\la^{(4)} &= ((w-4),(2,1^2),\varnothing).
\end{align*}

Arguing as before, removing the first row yields partitions in the block $B([1,s_1-1,s_2], 4)$, which is Scopes equivalent to $B([1,s_1-s_2+3,4], 4)$.
In turn, this is Scopes equivalent to $B([1,4,6], 4)$ if $s_1 - s_2=4$ or $B([1,4,7], 4)$ if $s_1 - s_2\geq5$.
The remaining partitions are precisely those used in both of \cref{prop:NearRouqblocks,prop:Rouqblocks}, and the result follows.
\end{proof}

\bibliographystyle{amsalpha}  
\addcontentsline{toc}{section}{\refname}
\bibliography{master}

\end{document}